\documentclass[12pt]{amsart}
\usepackage{amsmath,amssymb,amsfonts,amsthm,amsopn,dsfont}
\usepackage[dvips]{graphicx}
\usepackage{srcltx}
\usepackage{setspace}
\newtheorem{theorem}{Theorem}[section]

\newtheorem{lemma}[theorem]{Lemma}
\numberwithin{equation}{section}

\renewcommand{\ell}{l}
\renewcommand{\epsilon}{\varepsilon}

\def\R{\right)}

\def\<{\left<}
\def\>{\right>}

\def\mv1{M_v^1}

\def\mn{(m,n)}
\def\mn'{(m',n')}

\def\R{\mathbb{R}}

\begin{document}

\title[Spatial analyticity for the Euler equations]{Some remarks on the radius of spatial analyticity for the Euler equations}
\begin{abstract} 
We consider the Euler equations on $\mathbb{T}^d$ with analytic data and prove lower bounds for the radius  of spatial analyticity $\epsilon(t)$ of the solution
using a new method based on inductive estimates in standard Sobolev spaces. 
Our results are consistent with similar previous results proved by Kukavica and Vicol, but give a more precise dependence of $\epsilon(t)$ on the radius of analyticity
of the initial datum.
\end{abstract}
\author{Marco Cappiello \and Fabio Nicola}
\address{Dipartimento di Matematica,  Universit\`{a} degli Studi di Torino,
Via Carlo Alberto 10, 10123
Torino, Italy 
} \email{marco.cappiello@unito.it}
\address{Dipartimento di Scienze Matematiche, Politecnico di
Torino, Corso Duca degli
Abruzzi 24, 10129 Torino,
Italy}
\email{fabio.nicola@polito.it}
\date{}
\maketitle

\section{Introduction}
Consider, on the $d$-dimensional torus $\mathbb{T}^d$, the Euler equations
\begin{equation}\label{euler}
\frac{\partial u}{\partial t}+P(u\cdot\nabla_x u)=0,\quad {\rm div}\, u=0,
\end{equation}
where $P$ is the Leray projection in $L^2$ on the subspace of divergence free vector fields.
It is well known that the corresponding Cauchy problem is locally well posed in $H^k$ if $k>d/2+1$, see e.g. \cite[Chapter 17, Section 2]{taylor}. The analyticity of the solution in the space variables, for analytic initial data is also an important issue, investigated in \cite{AM, B, BB, BBZ,KV, KV2, LO}, and one is specially interested in lower bounds for the radius of analiticity $\epsilon(t)$ as $t$ grows.\par
To be precise, if $f$ is an analytic function on the torus, its radius of analyticity is the supremum of the constants $\epsilon>0$ such that $\|\partial^\alpha f\|_{L^\infty}\leq C\epsilon^{-|\alpha|} |\alpha|!$ for some constant $C>0$. 
 Notice that we can also replace the $L^\infty$ norm with a Sobolev norm $H^k$, $k\geq0$.\par
Concerning the Euler equations on the torus, a recent result by Kukavica and Vicol \cite{KV} states that for the radius of analyticity $\epsilon(t)$ of any analytic solution $u(t)$ we have the lower bound
\begin{equation} \label{radius1}
\epsilon(t)\geq C (1+t)^{-2}\exp\Big( -C_0\int_0^t \|\nabla u(s)\|_{L^\infty}\, ds\Big)
\end{equation}
for a constant $C_0>0$ depending on the dimension and $C>0$ depending on the norm of the initial datum in some \textit{infinite order} Sobolev space. \par
The same authors in \cite{KV2} obtained a better lower bound for $\epsilon(t)$ for the Euler equations in a half space replacing $(1+t)^{-2}$ by $(1+t)^{-1}$ in \eqref{radius1}. Now, one suspects that the same bound should hold also on $\mathbb{T}^d$ and on $\R^d$. In this paper we give a new and more elementary proof of the results above on $\mathbb{T}^d$, which yields some improvements concerning the dependence on the initial data. In fact, it is natural to expect that the constant $C$ in \eqref{radius1} should be comparable with the inverse of the radius of analyticity of the initial datum. The dependence found in \cite{KV, KV2} is due to the fact that the proof given there relies on the energy method in infinite order Gevrey-Sobolev spaces (cf. also \cite{ bo1, bo2, FT, GK, LO}). In our recent paper \cite{CDN} we developed a method for the estimate of the radius of analyticity for semilinear symmetrizable hyperbolic systems based on inductive estimates in standard Sobolev spaces. The purpose of this note is to adapt this method to the Euler equations on $\mathbb{T}^d$ and to prove that
\begin{equation} \label{radius2}
\epsilon(t)\geq C (1+t)^{-1}\exp\Big( -C_0 \int_0^t \|\nabla u(s)\|_{L^\infty}\, ds\Big)
\end{equation}
(as suggested in \cite{KV2}) with a neat dependence of the constant $C$ on the initial datum. \par
Namely, we have the following result. 
\begin{theorem}\label{mainteo}
Let $k>d/2+1$ be fixed. There exists constants $C_0,C_1>0$, depending only on $k$ and $ d$, such that the following is true.\par Let $u_0$ be analytic in $\mathbb{T}^d$, ${\rm div}\, u_0=0$, satisfying\footnote{We use the sequence $|\alpha|!/(|\alpha|+1)^2$ in place of $|\alpha|!$ in \eqref{ipotesi0} just for technical reasons; this does not change the radius of analyticity.} 
\begin{equation}\label{ipotesi0}
\|\partial^\alpha u_0\|_{H^k}\leq B A^{|\alpha|-1}|\alpha|!/(|\alpha|+1)^2,\quad\alpha\in\mathbb{N}^d,
\end{equation}
 for some $B\geq \frac{9}{4}\|u_0\|_{H^{2k+1}}$, $A\geq 1$.\par Let $u(t,x)$ by the corresponding $H^k$ maximal solution of the Euler equations \eqref{euler}, with $u(0,\cdot)=u_0$. 
 Then $u(t,\cdot)$ is analytic with radius of analyticity
\begin{equation}
\epsilon(t)\geq A^{-1} (1+C_1Bt)^{-1}\exp\Big(-C_0\int_0^t \|\nabla u(s)\|_{L^\infty}\, ds\Big).
\end{equation}
\end{theorem} 

The proof is different and more elementary than the one in \cite{KV, KV2}, and it is in part inspired by the arguments in \cite{AM}. It proceeds by estimating by induction the growth of the spatial derivatives of $u$ in {\it finite order} Sobolev spaces. Moreover, the same argument can be readily repeated replacing $\mathbb{T}^d$ by $\R^d$. Although in the case of the Euler equations it provides only minor improvements to the results in \cite{KV, KV2}, our method seems to be adaptable also to other types of quasilinear evolution equations and conservation laws. We shall treat these applications in a future paper.


\section{Notation and preliminary results}
In the following we use the notation $X\lesssim Y$ if $X\leq C Y$ for some constant $C$ {\it depending only on the dimension $d$ and on the index $k$ in Theorem \ref{mainteo}}.\par
Moreover, as in \cite[page 196]{AM} we consider the sequence 
\begin{equation}\label{mn}
M_n=\frac{n!}{(n+1)^2},\quad n\geq0,
\end{equation}
so that
\begin{equation}\label{binomiali}
\sum_{\beta<\alpha}\binom{\alpha}{\beta}M_{|\alpha-\beta|}M_{|\beta|+1}\leq  C|\alpha|M_{|\alpha|}.
\end{equation}
for some constant $C>0$. \par
We also recall from \cite[Chapter 13, Proposition 3.6]{taylor}, for future reference, the following estimates
\begin{equation}\label{prodotto}
\|\partial^\alpha u \cdot \partial^\beta v\|_{L^2}\leq C(\|u\|_{L^{\infty}}\|v\|_{H^m}+\|u\|_{H^m}\|v\|_{L^{\infty}}) ,\quad |\alpha|+|\beta|=m
\end{equation}
for a constant $C>0$ depending on $m$ and on the dimension $d$. \par
We will use the following form of the Gronwall inequality. 
\begin{lemma}\label{gronwall} Let $f(t)\geq0$, $g(t)\geq0$, $h(t)\geq0$ be continuous functions on $[0,T]$ and $C\geq 0$, such that
\[
f(t)\leq C+\int_0^t h(s)f(s)\, ds+\int_0^t g(s)\, ds, \quad t\in[0,T].
\]
Then, with $H(t):=\int_0^t h(s)\, ds$, we have
\[
f(t)\leq e^{H(t)}\Big[C+\int_0^t e^{-H(s)} g(s)\, ds\Big], \quad t\in[0,T].
\]
\end{lemma}
\begin{proof}
The result can be obtained for example by applying Gronwall lemma in \cite[Lemma 2.1.3]{rauch}, and integrating by parts.
\end{proof}

Finally we recall from \cite[Chapter 17, Section 2 and Exercise 1, page 485]{taylor} that if $u_0$ is a smooth vector field (and ${\rm div}\, u_0=0$), then the maximal $H^k$-solution $u(t)$ of the Euler equations, with $u(0)=u_0$, $k>d/2+1$, is smooth as well and moreover the following estimates hold for its Sobolev norms: if $s\geq k>d/2+1$,
\begin{equation}\label{sobolev}
\|u(t)\|_{H^s}\leq \|u(0)\|_{H^s} \exp\Big(C_0 \int_0^t \|\nabla u(s)\|_{L^\infty}\, ds\Big)
\end{equation}
for some constant $C_0>0$ depending on the dimension and on $s$. Indeed, in the sequel we will use these estimates for some fixed $s$ (depending only on $d$). 

\section{Proof of the main result (Theorem \ref{mainteo})}
As observed in the previous section, we already know that the solution $u$ is smooth, since $u_0$ is. Now, it is sufficient to prove that for $|\alpha|=N\geq 2$ we have 
\begin{equation}\label{ipotesi}
\frac{\|\partial^\alpha u(t)\|_k}{M_{|\alpha|}}\leq 2B A^{N-1} \exp\Big(C_0(N-1)\int_0^t \|\nabla u(s)\|_{L^\infty}\, ds\Big)(1+C_1Bt)^{N-2},
\end{equation}
 where the sequence $M_{|\alpha|}$ is defined by \eqref{mn} and $C_0,C_1$ are positive constants depending only on $k$ and $d$.\par
We also set
\begin{equation}\label{art}
\mathcal{E}_{N}[u(t)]=\sup_{|\alpha|=N}\frac{\|\partial^{\alpha} u(t)\|_k}{M_{|\alpha|}}.
\end{equation}
We proceed by induction on $N$. The result is true for $N=2$ by \eqref{sobolev} with $s=k+2 \leq 2k+1$ and by the assumption $B\geq \frac{9}{4}\|u_0\|_{H^{2k+1}}$, $A\geq 1$. Hence, let $N\geq 3$ and assume \eqref{ipotesi} holds for multi-indices $\alpha$ of length $2 \leq |\alpha| \leq N-1$ and prove it for $|\alpha|=N$.\par
For $|\alpha|=N$, $|\gamma|\leq k$ we estimate $\|\partial^{\alpha+\gamma} u\|_{L^2}$ starting from the following formula, which is well-known (see e.g. \cite[pag. 477]{taylor}):
\begin{equation}\label{prima}
\frac{d}{dt}\|\partial^{\alpha+\gamma} u\|^2_{L^2}=-2([\partial^{\alpha+\gamma},L] u,\partial^{\alpha+\gamma} u)_{L^2},
\end{equation}
with $Lw=L_u w:=u\cdot\nabla w$. Now, we have 
\begin{equation}\label{commutatore}
[\partial^{\alpha+\gamma},L]u=\sum_{\beta\leq\alpha}\binom{\alpha}{\beta}\sum_{\delta\leq\gamma\atop |\beta|+|\delta|<|\alpha|+|\gamma|}\binom{\gamma}{\delta}\partial^{\alpha-\beta+\gamma-\delta} u\cdot\nabla \partial^{\beta+\delta} u.
\end{equation}
We estimate the $L^2$ norm of each term, considering first the sum
\[
\sum_{\beta\in \mathcal{A}_{\alpha} }\binom{\alpha}{\beta}\sum_{\delta\leq\gamma\atop |\beta|+|\delta|<|\alpha|+|\gamma|}\binom{\gamma}{\delta}\| \partial^{\alpha-\beta+\gamma-\delta} u\cdot\nabla \partial^{\beta+\delta} u\|_{L^2},
\]
where
\[
\mathcal{A}_{\alpha}:=\{\beta:\,\beta\leq\alpha, 0\not=|\beta|\leq |\alpha|-2\}.
\]
Using \eqref{prodotto} and the fact that $k >d/2+1$ we see that for $\beta \in \mathcal{A}_{\alpha}$ we have
$$
\| \partial^{\alpha-\beta+\gamma-\delta} u\cdot\nabla \partial^{\beta+\delta} u\|_{L^2}
\lesssim \| \partial^{\alpha-\beta}u\|_{k} \|\nabla \partial^{\beta} u\|_{k}.
$$
By the inductive hypothesis \eqref{ipotesi} (note that $2\leq |\alpha-\beta|\leq |\alpha|-1=N-1$, $2\leq |\beta|+1\leq N-1$) we obtain
\begin{multline*}
\sum_{\beta\in \mathcal{A}_{\gamma,\delta,\alpha} }\binom{\alpha}{\beta}\sum_{\delta\leq\gamma}\binom{\gamma}{\delta}\| \partial^{\alpha-\beta+\gamma-\delta} u\cdot\nabla \partial^{\beta+\delta} u\|_{L^2} \\
\lesssim \sum_{\beta<\alpha}\binom{\alpha}{\beta}M_{|\alpha-\beta|}M_{|\beta|+1} B^2 A^{N-1} \exp\Big(C_0(N-1)\int_0^t \|\nabla u(s)\|_{L^\infty}\, ds\Big)(1+C_1Bt)^{N-3}\\
\lesssim 
NM_{N} B^2 A^{N-1} \exp\Big(C_0(N-1)\int_0^t \|\nabla u(s)\|_{L^\infty}\, ds\Big)(1+C_1Bt)^{N-3},
\end{multline*}
where we used \eqref{binomiali}. \par
It remains to estimate  the $L^2$ norms of the terms in \eqref{commutatore} when $\beta\leq\alpha$, $\delta\leq\gamma$ and $ |\beta|+|\delta|<|\alpha|+|\gamma|$, but the conditions
\[
0\not=|\beta|\leq |\alpha|-2
\]
fail.\par
Consider first the terms where the highest order derivatives fall on a single factor, namely $|\beta|+|\delta|=|\alpha|+|\gamma|-1$ or $|\beta|+|\delta|=0$. We distinguish three cases: for the terms with $\beta=\alpha$ and $|\delta|=|\gamma|-1$ we have  
\[
\binom{\alpha}{\alpha}\|\partial^{\gamma-\delta} u\cdot\nabla \partial^{\alpha+\delta} u\|_{L^2} \lesssim \|\nabla u\|_{L^\infty}\|\partial^\alpha u\|_{H^{|\delta|+1}}\leq \|\nabla u\|_{L^\infty}\|\partial^\alpha u\|_{H^k}
\]
whereas for those with $|\beta|=|\alpha|-1$, $\delta=\gamma$ we use
\begin{align*}
\binom{\alpha}{\beta}\|\partial^{\alpha-\beta}& u\cdot\nabla \partial^{\beta+\gamma} u\|_{L^2} 
\lesssim |\alpha| \|\nabla u\|_{L^\infty}\sum_{j:\alpha_j\geq 1}\|\partial^{\alpha-e_j} u\|_{H^{k+1}}\\
&\lesssim  |\alpha| \|\nabla u\|_{L^\infty}\sum_{j:\alpha_j\geq 1\atop 1\leq k\leq d}\|\partial^{\alpha-e_j+e_k} u\|_{H^{k}}\lesssim N M_{N}\|\nabla u\|_{L^\infty}\mathcal{E}_{N}[u],
\end{align*}
where $\mathcal{E}_{N}[u]$ is defined in \eqref{art}. \par
Finally, for the terms with $|\beta|+|\delta|=0$ we have
\[
\binom{\alpha}{0}\|\partial^{\alpha+\gamma}u\cdot\nabla u\|_{L^2} 
 \lesssim \|\nabla u\|_{L^\infty}\|\partial^\alpha u\|_{H^k}.
\]
We now consider the terms with $0\not=|\beta|+|\delta|\leq|\alpha|+|\gamma|-2$ but $\beta=\alpha$ or $|\beta|=|\alpha|-1$ or $\beta=0$.\par
If $\beta=\alpha$ then $|\delta|\leq |\gamma|-2\leq k-2$ and we can write 
\begin{align*}
\binom{\alpha}{\alpha}&\|\partial^{\gamma-\delta} u\cdot \nabla \partial^{\alpha+\delta} u\|_{L^2}\leq \|\partial^{\gamma-\delta} u\|_{L^\infty}\|\nabla \partial^{\alpha+\delta} u\|_{L^2}\\
&\lesssim \|\partial^{\gamma-\delta} u\|_{H^k}\| \partial^{\alpha} u\|_{H^{k-1}}\lesssim \|\partial^{\gamma-\delta} u\|_{H^k}\sup_{j:\alpha_j\geq 1}\|\partial^{\alpha-e_j} u\|_{H^{k}}\\
&\lesssim M_{N-1}\|u(0)\|_{H^{2k}}\exp\Big(C_0 \int_0^t \|\nabla u(s)\|_{L^\infty}\, ds\Big) \\ 
&\qquad\qquad\qquad \times BA^{N-2}\exp\Big(C_0(N-2)\int_0^t \|\nabla u(s)\|_{L^\infty}\, ds\Big)(1+C_1Bt)^{N-3}\\
&\lesssim M_{N-1}B^2A^{N-2}\exp\Big(C_0(N-1)\int_0^t \|\nabla u(s)\|_{L^\infty}\, ds\Big)(1+C_1Bt)^{N-3},
 \end{align*}
where we used \eqref{sobolev} (with $s=|\gamma-\delta|+k\leq 2k$), the inductive hypothesis \eqref{ipotesi}, and the fact that $B>\|u_0\|_{H^{2k+1}}$.\par
If $|\beta|=|\alpha|-1$ then $|\delta|\leq k-1$ and we have
\begin{align*}
\binom{\alpha}{\beta}&\|\partial^{\gamma-\delta} u\cdot\nabla \partial^{\beta+\delta} u\|_{L^2}\lesssim |\alpha|\cdot \|\partial^{\gamma-\delta}u\|_{L^\infty}\|\nabla\partial^{\beta+\delta} u\|_{L^2}\lesssim |\alpha| \cdot \|\partial^{\gamma-\delta}u\|_{H^k}\|\partial^{\beta} u\|_{H^k}\\
&\lesssim N M_{N-1} \|u(0)\|_{H^{2k}}\exp\Big(C_0 \int_0^t \|\nabla u(s)\|_{L^\infty}\, ds\Big) \\ &\qquad\qquad\qquad \times BA^{N-2}\exp\Big(C_0(N-2)\int_0^t \|\nabla u(s)\|_{L^\infty}\, ds\Big)(1+C_1Bt)^{N-3}\\
&\lesssim NM_{N-1} B^2A^{N-2}\exp\Big(C_0(N-1)\int_0^t \|\nabla u(s)\|_{L^\infty}\, ds\Big)(1+C_1Bt)^{N-3}.
\end{align*}
Finally, if $\beta=0$ then $\delta\not=0$ and we have 
\begin{align*}
\binom{\alpha}{0}\|\partial^{\alpha+\gamma-\delta} u\cdot\nabla \partial^{\delta} u\|_{L^2}&\lesssim\|\partial^{\alpha} u\|_{H^{k-1}}\|\nabla \partial^{\delta} u\|_{H^k}\\
&\lesssim \sup_{j:\alpha_j\geq 1} \|\partial^{\alpha-e_j} u\|_{H^{k}}\|\nabla \partial^{\delta} u\|_{H^k}\\
&\lesssim M_{N-1} B^2A^{N-2}\exp\Big(C_0(N-1)\int_0^t \|\nabla u(s)\|_{L^\infty}\, ds\Big)\\ & \times (1+C_1Bt)^{N-3}.
\end{align*}
Summing up, we have
\begin{multline*}
\|[\partial^{\alpha+\gamma}, L] u\|_{L^2}
\lesssim N M_N\|\nabla u\|_{L^\infty} \mathcal{E}_{N}[u]\\ +NM_{N} B^2A^{N-1} \exp\Big(C_0(N-1)\int_0^t \|\nabla u(s)\|_{L^\infty}\, ds\Big)(1+C_1Bt)^{N-3}.
\end{multline*}
By applying the Cauchy-Schwarz inequality in $L^2$ in \eqref{prima} and summing over $|\gamma|\leq k$ we then obtain\footnote{We use $N/(N-1)\leq 3/2 $ if $N\geq 3$, and also  $\frac{d}{dt}\|\partial^{\alpha} u(t)\|^2_{H^k}=2\|\partial^{\alpha} u(t)\|_{H^k}\frac{d}{dt}\|\partial^{\alpha} u(t)\|_{H^k}$, which in fact holds for $\|\partial^{\alpha} u(t)\|_{H^k}\not=0$ but a standard argument, see e.g. \cite[pag. 47-48]{rauch}, shows that the results below still hold when $\|\partial^{\alpha} u(t)\|_{H^k}=0$.}
\begin{multline*}
\frac{d}{dt}\frac{\|\partial^{\alpha} u(t)\|_{H^k}}{M_{|\alpha|}} \leq C (N-1)\|\nabla u(t)\|_{L^\infty} \mathcal{E}_N[u(t)]\\ +CN A^{N-1}B^2 \exp\Big(C_0(N-1)\int_0^t \|\nabla u(s)\|_{L^\infty}\, ds\Big)(1+C_1Bt)^{N-3}
\end{multline*}
for a constant $C>0$ depending only on the dimension $d$ and $k$.\par 
Now we integrate from $0$ to $t$ and take the supremum on $|\alpha|=N$. We obtain 
\begin{align*}
&\mathcal{E}_N[u(t)]\leq \int_0^t C (N-1)\|\nabla u(s)\|_{L^\infty} \mathcal{E}_N[u(s)]\, ds +\mathcal{E}_N[u(0)]\\
&+\int_0^t CN B^2A^{N-1} \exp\Big(C_0(N-1)\int_0^s \|\nabla u(\tau)\|_{L^\infty}\, d\tau\Big)(1+C_1Bs)^{N-3}\, ds.
\end{align*}

 We can take $C_0\geq C$, so that 
Gronwall inequality (Lemma \ref{gronwall}) gives
\begin{align*}
\mathcal{E}_N[u(t)]&\leq\exp\Big(C_0(N-1)\int_0^t \|\nabla u(s)\|_{L^\infty}\, ds\Big) \\
&\times \Big[\mathcal{E}_N[u(0)]+CNB^2A^{N-1}\int_0^t (1+C_1Bs)^{N-3} \,ds\Big]\\
&\leq \exp\Big(C_0(N-1)\int_0^t \|\nabla u(s)\|_{L^\infty}\, ds\Big) \\
&\times \Big[\mathcal{E}_N[u(0)]+\frac{C}{C_1}\frac{N}{N-2}BA^{N-1}(1+C_1Bt)^{N-2}\Big]
\end{align*}
Now, we have $\mathcal{E}_N[u(0)]\leq  BA^{N-1}$ by the assumption \eqref{ipotesi0}. If we choose $C_1=3C$, so that $3C/C_1=1$, since $A\geq 1$, $N\geq3$  we have
\begin{align*}
\mathcal{E}_N[u(0)]+&\frac{C}{C_1}\frac{N}{N-2}BA^{N-1}(1+C_1Bt)^{N-2} \\
&\leq [BA^{N-1}+\frac{3C}{C_1} BA^{N-1}](1+C_1Bt)^{N-2}\\
&\leq 2BA^{N-1}(1+C_1Bt)^{N-2},
\end{align*}
and we obtain exactly \eqref{ipotesi} for $|\alpha|=N$. The theorem is proved.

\end{document}